\documentclass[12pt]{article}

\usepackage{amsmath,amssymb,amsthm,amsfonts}
\usepackage{nicematrix,tikz}
\usetikzlibrary{fit}
\tikzset{highlight/.style={rectangle,
fill=gray!50,
rounded corners = 0.5 mm, 
inner sep=1pt,
fit=#1}}

\def\Hom{{\rm Hom}}
\def\ind{{\rm ind}}
\def\ad{\mathop{\rm ad}}
\def\im{\mathop{\rm im}}

\def\phi{\varphi}
\def\g{\mathfrak g}

\def\h{\mathfrak h}
\def\F{\mathbb F}

\newtheorem{theorem}{Theorem}[section]
\newtheorem{lemma}[theorem]{Lemma}

\newtheorem{definition}[theorem]{Definition}

\theoremstyle{remark}

\providecommand{\keywords}[1]{\noindent{Keywords:} #1}
\providecommand{\classify}[1]{\noindent{Mathematics Subject
    Classification:} #1}

\title{On the Restricted Cohomology  of Twisted
Heisenberg Lie Superalgebras}

\author{
Yong Yang\\
College of Mathematics and System Science\\
Xinjiang University\\
Urumqi 830046, China\\
yangyong195888221@163.com
}

\date{}

\begin{document}
\maketitle

\begin{abstract}
Restricted twisted  Heisenberg Lie
  superalgebras are studied over an algebraically closed field $\F$ of
  characteristic $p>0$. We determine the restricted structures and use the ordinary 1- and 2-cohomology spaces with
  trivial coefficients to compute the restricted 2-cohomology
  spaces. As an application,  the corresponding 
  restricted one-dimensional central extensions are classified and described. 
\end{abstract}

{\footnotesize

\classify{17B30; 17B50; 17B56}}

\section{Introduction}

The theory of Lie superalgebras and Lie supergroups has many applications in various areas
of the modern mathematics and theoretical physics, for example, see \cite{A,BP,F,H,M}. Lie superalgebras over ﬁelds of characteristic $p>0$ are known as modular Lie superalgebras. 
The concept of a restricted Lie algebra is attributable to
Jacobson \cite{J}, who observed the importance of a $[p]$-operator having the properties
of the ``Frobenius map" $x\mapsto x^p$ of an associative algebra.
While Lie algebras over ﬁelds of characteristic zero are closely related to Lie and
algebraic groups, there is no similar correspondence for modular ones. 
The Lie algebras associated with algebraic groups are restrictable.
Besides that,
Milner' s theorem \cite{SF} states that every modular Lie algebra can be embedded into its universal $p$-envelope as a Lie subalgebra.
Therefore, restricted Lie algebras play important roles in the study of modular Lie algebras.
A restricted Lie superalgebra is a Lie superalgebra having a restricted Lie algebra as its even part and a restricted module as its odd part. 
Thus, restricted Lie superalgebras can be viewed as a generalization of restricted Lie algebras in the category of superalgebras.
In 1954, Hochschild \cite{H} established 
the restricted cohomology theory
of restricted Lie algebras and  associated the usual Chevalley-Eilenberg cohomology
 with the restricted cohomology by the canonical projection from its universal enveloping algebra to the restricted universal enveloping algebra. Then
 Evans and Fuchs constructed a cochain complex for restricted cohomology
 up to order 3 in \cite{EFu}. It has been proved to be an effective tool to compute restricted cohomology, for example, see \cite{EFi2,EFP,EFY}. Recently, the
 restricted cohomology theory of restricted Lie superalgebras was introduced
 by Bouarroudj and Ehret \cite{BE}.

In this paper, we concentrate on restricted twisted Heisenberg Lie superalgebras, defined as follows:

\begin{definition}\rm\cite{X}
For non-negative integers $m,n,t$, the \emph{twisted Heisenberg Lie superalgebra}
 $\mathfrak{h}^{\lambda,\kappa}_{m,n,t}$, parameterized by $\lambda=(\lambda_1,\ldots,\lambda_m)\in(\F^{\times})^{m}$ and $\kappa=(\kappa_1,\ldots,\kappa_n)\in(\F^{\times})^{n}$, is the $(2m+2,2n+t)$-dimensional vector superspace over a
 field $\F$ spanned by the elements, where the even basis elements first separated from the odd ones by a vertical bar:
 \[\{e_1,\ldots , e_{2m},e_{2m+1},e_{2m+2}\mid \omega_1,\ldots,\omega_{2n},\eta_1,\ldots,\eta_t\}\] with the non-vanishing Lie brackets
\begin{eqnarray*}
 &&[e_i,e_{m+i}]=[\omega_j,\omega_j]=-[\omega_{n+j},\omega_{n+j}]=[\eta_k,\eta_k]=e_{2m+1},\\
 &&[e_{2m+2},e_i]=\lambda_ie_{m+i},\ [e_{2m+2},e_{m+i}]=\lambda_ie_i,\\
&&[e_{2m+2},\omega_j]=\kappa_j\omega_{n+j},\ [e_{2m+2},\omega_{n+j}]=\kappa_j\omega_{j},
 \end{eqnarray*}
 for $1\le i\le m$, $1\le j\le n$ and $1\le k\le t$.
\end{definition}

The even parts of these Lie superalgebras are known as twisted Heisenberg Lie algebras, which are  certain one-dimensional extensions of the corresponding Heisenberg Lie algebras as
 shown in the following definition:
 
\begin{definition}\rm\cite{T}
For a non-negative integer $m$, the \emph{twisted Heisenberg Lie algebra}
 $\mathfrak{h}^{\lambda}_{m}$, parameterized by $\lambda=(\lambda_1,\ldots,\lambda_m)\in(\F^{\times})^{m}$, is the $(2m+2)$-dimensional vector space over a
 field $\F$ spanned by the elements
 \[\{e_1,\ldots , e_{2m},e_{2m+1},e_{2m+2}\}\] with the non-vanishing Lie brackets
 \[[e_i,e_{m+i}]=e_{2m+1},\ [e_{2m+2},e_i]=\lambda_ie_{m+i},\ [e_{2m+2},e_{m+i}]=\lambda_ie_i,\] for $1\le i\le m$.
\end{definition}
Due to the relation between twisted Heisenberg Lie algebras 
 and Lorentzian manifolds \cite{A,G}, there are many studies on these algebras \cite{A,T,Y1}. In \cite{Y1}, we determined all restricted structures on a twisted Heisenberg Lie algebra and characterized the restricted cohomology and restricted  extensions.
However, nothing is known about restricted twisted Heisenberg Lie superalgebras. These superalgebras are solvable since the derived superalgebras are two-step nilpotent with one-dimensional even centers. This class of nilpotent Lie superalgebras are called Heisenberg Lie superalgebras with even centers: 

\begin{definition}\rm\cite{R}
For non-negative integers $m$ and $n$, the \emph{Heisenberg Lie superalgebra with the even center}
$\mathfrak{h}_{m,n}$ is the $(2m+1,n)$-dimensional vector superspace over a
field $\F$ spanned by the elements,
where the even basis elements first separated from the odd
ones by a vertical bar: 
 \[\{e_1,\ldots , e_{2m},e_{2m+1} \mid \omega_1,\ldots,\omega_{n}\}\] with the non-vanishing Lie brackets
\[
[e_i,e_{m+i}]=[\omega_j,\omega_j]=e_{2m+1},\]
for $1\le i\le m$ and $1\le j\le n$.
\end{definition}

Generally, a two-step nilpotent Lie superalgebra with one-dimensional center is called a Heisenberg Lie superalgebra. In \cite{R}, all finite-dimensional
Heisenberg Lie superalgebras split precisely into two types according to the parity of the center and these superalgebras correspond to Heisenberg supergroups \cite{BP,CK}.
The ordinary and restricted cohomology  of Heisenberg Lie superalgebras with coefficients in the trivial module were computed in \cite{BaLi,Y2} both over fields of  characteristic zero and prime characteristic.

The paper is organized as follows.
In section 2, we recall the definitions of restricted Lie superalgebras and determine the restricted Lie superalgebra structures on the twisted Heisenberg Lie superalgebra
$\h^{\lambda,\kappa,\mu}_{m,n,t}$.
In secton 3, we recall Chevalley-Eilenberg cohomology for Lie superalgebras and the restricted Lie superalgebra cohomology for restricted Lie superalgebras. In section 4, we compute the ordinary 1- and 2-cohomology of $\h^{\lambda,\kappa,\mu}_{m,n,t}$ with coefficients in the trivial module by means of the Hochschild-Serre spectral sequence to the Heisenberg Lie superalgebra ideal. In section 5, we compute 
1- and 2-restricted cohomology of $\h^{\lambda,\kappa,\mu}_{m,n,t}$
with the Hochschild’s six-term exact sequence. In section 6, we classified and characterized the one-dimensional  restricted central
extensions.

\section{Restricted Lie superalgebra $\h^{\lambda,\kappa,\mu}_{m,n,t}$}

A \emph{restricted Lie algebra}  $\mathfrak{g}$ over a field $\F$ of positive characteristic $p$ is a Lie
algebra $\mathfrak{g}$ over $\F$ together with a map
$[p]:\g\to\g$, written $g\mapsto g^{[p]}$, such that for all
$a\in\F$, and all $g,h\in \mathfrak{g}$
\begin{itemize}
\item[(1)] $(a g)^{[p]}=a^{p}g^{[p]}$;
\item[(2)]
  $(g+h)^{[p]}=g^{[p]}+h^{[p]}+ \sum\limits_{i=1}^{p-1} s_i(g,h)$,
  where $is_i(g,h)$ is the coefficient of $t^{i-1}$ in the formal expression
  $(\mathrm{ad}(tg+h))^{p-1}(g)$; and
\item[(3)] $(\mathrm{ad}\ g)^{p}=\mathrm{ad}\ g^{[p]}$.
\end{itemize}

The map $[p]$ is called a \emph{$[p]$-operator} on
$\mathfrak{g}$.  Let $(\g,[p])$ be a restricted Lie algebra. A $\g$-module $M$ is called \emph{restricted} if, for $g\in\g$ and $m\in M$,
\[\underbrace{g\ldots g}_{p}\cdot m=g^{[p]}\cdot m.\]
\begin{definition}\rm\label{L}
A \emph{restricted Lie superalgebra} is a Lie superalgebra $\g=\g_{\bar{0}}\oplus \g_{\bar{1}}$ with a map $[p]:\g_{\bar{0}} \rightarrow \g_{\bar{0}}$ such that  $(\g_{\bar{0}},[p])$ is a restricted Lie algebra and $\g_{\bar{1}}$ is a restricted $\g_{\bar{0}}$-module with respect to the bracket.
\end{definition}
In \cite{Y1}, we studied the restricted twisted Heisenberg Lie algebras and proved the following theorem.
\begin{theorem}\label{Li}
For $m\ge 1$ and $\lambda\in(\F^{\times})^{m}$,
the twisted Heisenberg Lie algebra $\mathfrak{h}^{\lambda}_m$ is restricted if and only if $p>2$ and $\lambda^{p-1}_1=\cdots=\lambda^{p-1}_m$.
\end{theorem}

Suppose that $p > 2$ and $\lambda^{p-1}_1=\cdots=\lambda^{p-1}_m$. Denote by $|\lambda|=\lambda^{p-1}_1=\cdots=\lambda^{p-1}_m$.  
By Jacobson's theorem \cite{J}, a restricted Lie algebra $[p]$-operator on $\mathfrak{h}^{\lambda}_m$, parameterized by $\mu=(\mu_1,\dots, \mu_{2m+2})\in \F^{2m+2}$,
is given by
$$e_i^{[p]}=\mu_i e_{2m+1},\quad
e_{2m+2}^{[p]}=|\lambda|e_{2m+2}+\mu_{2m+2} e_{2m+1},$$
where $1\le i \le 2m+1$.
Moreover,  if $g=\sum^{2m+2}_{i=1} a_i e_i\in \mathfrak{h}_m^{\lambda}$,
then (see \cite{Y1})
\begin{align}\label{res}
      \begin{split}
  g^{[p]}&=a_{2m+2}^{p-1}|\lambda|\sum_{i=1}^{2m}a_ie_i+a_{2m+2}^{p}|\lambda|e_{2m+2}\\
  &+\left(\sum^{2m+2}_{i=1}a^{p}_i\mu_i+2^{-1}a_{2m+2}^{p-2}\sum_{i=1}^{m}\lambda_{i}^{p-2}(a_i^{2}-a_{m+i}^{2})\right) e_{2m+1}.
  \end{split}
\end{align}

\begin{theorem}
For $m,n,t\geq 1$, the twisted Heisenberg Lie superalgebra $\h^{\lambda,\kappa}_{m,n,t}$ is restricted if and only if $p>2$ and
$\lambda_{1}^{p-1}=\cdots=\lambda^{p-1}_m=\kappa^{p-1}_1=\cdots=\kappa^{p-1}_n$.
\end{theorem}
\begin{proof}
By definition \ref{L} and Theorem \ref{Li},  
$\h^{\lambda,\kappa}_{m,n,t}$ is restricted if and only if $p>2$, $\lambda^{p-1}_1=\cdots=\lambda^{p-1}_m$ and 
\begin{equation}\label{l}
[\underbrace{g,[\cdots[g,[g}_p,h]]\cdots]]=[g^{[p]},h] 
\end{equation}
 for $g\in (\h^{\lambda,\kappa}_{m,n,t})_{\bar{0}}$ and $h\in(\h^{\lambda,\kappa}_{m,n,t})_{\bar{1}}$. Set $g=\sum a_ie_i$ and $h=\sum b_j \omega_j+\sum c_k\eta_k$. 
 From Eq. (\ref{res}),  Eq.  (\ref{l}) holds if and only if 
\[(\ad e_{2m+2})^{p}\left(\sum b_j\omega_j\right)=|\lambda|\ad e_{2m+2}\left(\sum b_j\omega_j\right).\]
That is,
$\sum_{i=1}^{n}\kappa_i^p(b_i\omega_{n+i}+b_{n+i}\omega_i)=|\lambda|\sum_{i=1}^{n}\kappa_i(b_i\omega_{n+i}+b_{n+i}\omega_i)$.
From this, we get that $\kappa_i^{p-1}=|\lambda|$ for $1\le i\le n.$ The proof is complete.
\end{proof}

From now on, we assume $p > 2$ and $\lambda^{p-1}_1=\cdots=\lambda^{p-1}_m=\kappa^{p-1}_1=\cdots=\kappa^{p-1}_n=|\lambda|$.  
We let
$\mu=(\mu_1,\dots, \mu_{2m+2})$ and denote the corresponding restricted Lie superalgebra $\h^{\lambda,\kappa,\mu}_{m,n,t}$.


\section{Chevalley-Eilenberg and restricted Lie superalgebra cohomology}
In this section, we recall the theories of
the Chevalley-Eilenberg cochain complex and restricted cochain complex \cite{BE}.
 Everywhere in this
section, $\F$ denotes an algebraically closed field of characteristic
$p>2$ and $\g$ denotes a finite-dimensional Lie superalgebra over $\F$ with an
ordered basis $\{e_1,\ldots,e_m\mid e_{m+1},\ldots,e_{m+n}\}$.
Write $e^k$,  $e^{i,j}$ and $e^{u,v,w}$ for the dual vectors of the basis vectors $e_k\in\g$, $e_{i,j}=e_i\wedge e_j\in \wedge^{2}\g$ and $e_{u,v,w}=e_u\wedge e_v\wedge e_w\in\wedge^{3}\g$, respectively.  For a a superspace $V$, we denote by the super-dimension $\mathrm{sdim}\ V=(\mathrm{dim}\ V_{\bar{0}}, \mathrm{dim}\ V_{\bar{1}})$. If $V$ is a $\g$-module and $x\in\g$, we set $V^{x}=\{v\in V\mid x\cdot v=0\}$.
For $j\geq2$ and $g_1,\ldots, g_{j}\in\g$, we denote the $j$-fold bracket
$$[g_1, g_2, g_3,\ldots,g_j] = [[\ldots[[g_1, g_2], g_3],\ldots], g_j].
$$

\subsection{Chevalley-Eilenberg Lie superalgebra cohomology}

We only describe  the
Chevalley-Eilenberg cochain spaces $C^q(\g)=C^q(\g,\F)$
for $q=0,1,2,3$ and differentials $d^q:C^q(\g)\to C^{q+1}(\g)$
for $q=0,1,2$.
 For general definitions and theories, the readers are referred to  \cite{M}. Set
$C^0 (\g)=\F$ and $C^q (\g)= \wedge^q\g^*$ for $q=1,2,3$. 
The differentials $d^q:C^q (\g)\to C^{q+1}(\g)$ are defined for $\psi\in C^1 (\g)$,
$\phi\in C^2 (\g)$ and $g,h,f\in\g$ by\small
\begin{align*}
  d^0: C^0 (\g)\to C^1 (\g),  &\  d^0=0&\\
  d^1:C^1 (\g)\to C^2 (\g), &\   d^1(\psi)(g\wedge h)=\psi([g,h])&\\
  d^2:C^2 (\g)\to C^3 (\g), &\  d^2(\phi)(g\wedge h\wedge f)=\phi([g,h]\wedge f)-(-1)^{|f||h|}\phi([g,f]\wedge h)&\\
&\qquad \qquad\qquad  \qquad  \        +(-1)^{|g|(|h|+|f|)}\phi([h,f]\wedge g). &
\end{align*}
The maps $d^q$ satisfy $d^{q}d^{q-1}=0$ and 
$H^q(\g)=H^q(\g,\F)=\ker(d^q)/\im(d^{q-1})$ for $q=1,2$ . 

\subsection{Restricted Lie superalgebra cohomology}

In this subsection, we recall the definitions and results on the partial
restricted cochain complex given in \cite{BE} only for the case of trivial
coefficients. 

Given $\phi\in C^2(\g)$, a map $\omega:\g_{\bar{0}}\to\F$ is \emph{
  $\phi$-compatible} if for all $g,h\in\g_{\bar{0}}$ and all $a\in\F$
\\
  
$\omega(a g)=a^p \omega (g)$ and
\begin{equation}
  \label{starprop}
  \omega(g+h)=\omega(g)+\omega(h) + \sum_{\substack{g_i=\mbox{\rm\scriptsize $g$
        or $h$}\\ g_1=g, g_2=h}}
  \frac{1}{\#(g)}\phi([g_1,g_2,g_3,\dots,g_{p-1}]\wedge g_p)
\end{equation}
where $\#(g)$ is the number of factors $g_i$ equal to $g$.

For $\phi\in C^2(\g)$, we can assign the values of $\omega$ arbitrarily on a
basis for $\g_{\bar{0}}$ and use (1) to define
$\omega: \g_{\bar{0}} \to \F $ that is $\phi$-compatible.
Since the sum (\ref{starprop}) is symmetric in $g$ and $h$, both $\varphi$ and the bracket are bilinear, and the exterior algebra is associative, the map $\omega$ is well-defined and unique (c.f. \cite{EFi2,EFu}).
 In
particular, given $\phi$, we can define $\tilde\phi(e_i)=0$ for all $i$ and use
(\ref{starprop}) to determine a unique $\phi$-compatible map
$\widetilde\phi:\g_{\bar{0}}\to\F$.  Moreover, If $\phi_1,\phi_2\in C^2(\g)$
and $a\in\F$, then
$\widetilde{(a\phi_1+\phi_2)} = a\widetilde\phi_1 + \widetilde\phi_2$.
\\

If $\zeta\in C^3(\g)$, then a map $\eta:\g\times \g_{\bar{0}}\to\F$ is \emph{
  $\zeta$-compatible} if for all $a\in\F$ and all $g\in\g,h,h_1,h_2\in\g_{\bar{0}}$,
$\eta(\cdot,h)$ is linear in the first coordinate,
$\eta(g,a h)=a^p\eta(g,h)$ and
\begin{align*}
  \eta(g,h_1+h_2) &=
                    \eta(g,h_1)+\eta(g,h_2)-\nonumber \\
                  & \sum_{\substack{l_1,\dots,l_p=1 {\rm or} 2\\ l_1=1,
  l_2=2}}\frac{1}{\#\{l_i=1\}}\zeta (g\wedge
  [h_{l_1},\cdots,h_{l_{p-1}}]\wedge h_{l_{p}}).
\end{align*}

The restricted cochain spaces are defined as $C^0_*(\g)=C^0 (\g)$,
$C^1_*(\g)=C^1 (\g)$,
\[C^2_*(\g)=\{(\phi,\omega)\ |\ \phi\in C^2 (\g), \omega:\g_{\bar{0}}\to\F\
  \mbox{\rm is $\phi$-compatible}\}\]
\[C^3_*(\g)=\{(\zeta,\eta)\ |\ \zeta\in C^3 (\g),
  \eta:\g\times\g_{\bar{0}}\to\F\ \mbox{\rm is $\zeta$-compatible}\}.\] 

We define
\begin{align*}
\Hom_{\rm Fr}(\g_{\bar{0}},\F) =& \{f:\g_{\bar{0}}\to\F\ |\ f(a x+b
  y)=a^pf(x)+b^pf(y)\ \mathrm{for\ all}\ a,b\in\F\\
&\ \mathrm{and}\  x,y\in \g_{\bar{0}}\}
\end{align*}
 to be
the space of \emph{ Frobenius homomorphisms} from $\g_{\bar{0}}$ to $\F$.
For $1\le i\le m$, define $\overline e^i:\g_{\bar{0}}\to\F$ by
$\overline e^i \left(\sum_{j=1}^m a_j e_j\right ) = a_i^p.$ The set
$\{\overline e^i\ |\ 1\le i\le m\}$ is a basis for the space of Frobenius homomorphisms
$\Hom_{\rm Fr}(\g_{\bar{0}},\F)$. 

Note that a map
$\omega:\g_{\bar{0}}\to\F$ is $0$-compatible if and only if
$\omega\in \Hom_{\rm Fr}(\g_{\bar{0}},\F)$, then we have the exact sequence
\begin{align*}
  0 \longrightarrow   \Hom_{\rm Fr}(\g_{\bar{0}},\F)\stackrel{\iota}{\longrightarrow}  C^2_*(\g) \stackrel{\pi}{\longrightarrow} C^2(\g)  \longrightarrow 0.
\end{align*}
Therefore
$\dim C^2_*(\g) =\dim C^2(\g)+\dim\g_{\bar{0}}$ and
\[
  \{(e^{i,j},\widetilde{e^{i,j}})\ |\ 1\le i<j \le m\} \cup \{(e^{i,j},\widetilde{e^{i,j}})\ |\ 1\le i\le m, m+1\leq j \le m+n\}\]
\[\cup
\{(e^{i,j},\widetilde{e^{i,j}})\ |\ m+1\le i\le j \le m+n\}
\cup \{(0,\overline e^i)\ |\ 1\le i\le m\}\]
is a basis for $C^2_*(\g)$. We will use this basis in all computations that
follow.

Define $d_*^0=d^0$. For $\psi\in C^1_*(\g)$, define the map
$\ind^1(\psi):\g_{\bar{0}}\to\F$ by
 \[\ind^1(\psi)(g)=\psi(g^{[p]}).\] 
The map
$\ind^1(\psi)$ is $d^1(\psi)$-compatible for all $\psi\in C^1_*(\g)$,
and the differential $d^1_*:C^1_*(\g)\to C^2_*(\g)$ is defined by
\begin{equation}
  d^1_*(\psi) = (d^1(\psi),\ind^1(\psi)).
\end{equation}
By definition,
\begin{equation}\label{h1}
 H^1_*(\g)=(\g/([\g,\g]+\langle \g_{\bar{0}}^{[p]}\rangle_\F))^*.
\end{equation}

For $(\phi,\omega)\in C^2_*(\g)$, define the map
$\ind^2(\phi,\omega):\g \times\g_{\bar{0}} \to\F$ by the formula
\[\ind^2(\phi,\omega)(g,h)=\phi(g\wedge h^{[p]})-\varphi([g,\underbrace{h,\ldots,h}_{p-1}],h).\]
The map $\ind^2(\phi,\omega)$ is $d^2(\phi)$-compatible for all
$\phi\in C^2(\g)$, and the differential $d^2_*:C^2_*(\g)\to C^3_*(\g)$
is defined by
\begin{equation}
  d^2_*(\phi,\omega) =
  (d^2(\phi),\ind^2(\phi,\omega)). 
\end{equation}
Note that  if $\omega_1$ and $\omega_2$ are
both $\phi$-compatible, then
$\ind^2(\phi, \omega_1)=\ind^2(\phi, \omega_2)$.

These maps $d_*^q$ satisfy $d_*^{q}d_*^{q-1}=0$ and we define
\[H_*^q(\g)=H_*^q(\g,\F)=\ker(d_*^q)/\im(d_*^{q-1}) \]
for $q=1,2$.

\section{The cohomology $H^1(\h^{\lambda,\kappa,\mu}_{m,n,t})$ and $H^2(\h^{\lambda,\kappa,\mu}_{m,n,t})$}

For a Lie superalgebra $\g$,
a useful tool for computing the cohomology is the Hochschild-Serre spectral
sequence relative to an ideal $I\triangleleft\g$ and the second term is $E_{2}^{r,s}=H^{r}(\g/I,H^{s}(I))\Longrightarrow H^{r+s}(\g)$ \cite{M}.
For $\g=\h^{\lambda,\kappa,\mu}_{m,n,t}$,
we consider the Heisenberg Lie superalgebra $\h_{m,2n+t}=\mathrm{span}\{e_1,\ldots,e_{2m+1}\mid \omega_1,\ldots,\omega_{2n},\eta_1,\ldots,\eta_t\}$ as an ideal of $\h^{\lambda,\kappa,\mu}_{m,n,t}$.
Notice that 
\[
E_{2}^{r,s}=H^{r}(\h^{\lambda,\kappa,\mu}_{m,n,t}/\h_{m,2n+t},H^{s}(\h_{m,2n+t}))=
\left\{
                    \begin{array}{ll}
                      H^r(\F e_{2m+2},H^s(\h_{m,2n+t})), & \hbox{$r=0,1$;} \\
                      0, & \hbox{otherwise.}
                    \end{array}
                  \right.
\]
Moreover,
$E_{\infty}^{r,s}=E_{2}^{r,s}$, $r=0$ or 1; $E_{\infty}^{r,s}=0$, otherwise.

\begin{lemma}\label{SE}
Suppose that $k\geq0$. Then
\[
H^k(\h^{\lambda,\kappa,\mu}_{m,n,t}) =H^k(\h_{m,2n+t})^{e_{2m+2}}\oplus \left(\F e^{2m+2}\wedge \frac{H^{k-1}(\h_{m,2n+t})}{e_{2m+2}\cdot H^{k-1}(\h_{m,2n+t})}\right).
\]
\end{lemma}
\begin{proof}
For $k\geq 0$, we have
\[
H^k(\h^{\lambda,\kappa,\mu}_{m,n,t}) =\bigoplus_{r+s=k}E^{r,s}_{\infty}
=H^0(\F e_{2m+2},H^k(\h_{m,2n+t}))\oplus H^1(\F e_{2m+2},H^{k-1}(\h_{m,2n+t})).\]
By the definitions of low cohomology \cite{M}, we have
\begin{eqnarray*}
&&H^0(\F e_{2m+2},H^k(\h_{m,2n+t}))=
H^k(\h_{m,2n+t})^{e_{2m+2}},\\
&&H^1(\F e_{2m+2},H^{k-1}(\h_{m,2n+t}))=\F e^{2m+2}\wedge \frac{H^{k-1}(\h_{m,2n+t})}{e_{2m+2}\cdot H^{k-1}(\h_{m,2n+t})}.
\end{eqnarray*}
Thus, we obtain that
\[
H^k(\h^{\lambda,\kappa,\mu}_{m,n,t}) =H^k(\h_{m,2n+t})^{e_{2m+2}}\oplus \left(\F e^{2m+2}\wedge \frac{H^{k-1}(\h_{m,2n+t})}{e_{2m+2}\cdot H^{k-1}(\h_{m,2n+t})}\right).
\]
The proof is complete.
\end{proof}
In order to compute the cohomology of $\h^{\lambda,\kappa,\mu}_{m,n,t}$, we here recall some results on the cohomology of the Heisenberg Lie superalgebra $\h_{m,2n+t}$ in \cite{BaLi}. It is easy to check these results hold for both $\mathrm{char}\ \F=0$ and $\mathrm{char}\ \F=p>2$.
\begin{lemma}\label{se}
Suppose that $\h_{m,2n+t}$ is the Heisenberg Lie superalgebra spanned by $\{e_1,\ldots,e_{2m+1}\mid \omega_1,\ldots,\omega_{2n},\eta_1,\ldots,\eta_t\}$. Then

(1) the space $H^1(\h_{m,2n+t})$ is spanned by  the classes of the cocycles
\[\{e^1,\ldots,e^{2m}\mid \omega^1,\ldots,\omega^{2n},\eta^1,\ldots,\eta^t\},\]

(2) the space $H^2(\h_{m,2n+t})$ is spanned by  the classes of the cocycles
 \[\{e^{i,j}
   \mid 1\le i< j\le 2m\}
\cup \{e^{i}\wedge \omega^j,e^{i}\wedge \eta^k, \omega^j\wedge \eta^k \mid 1\le i\le 2m, 1\le  j\le 2n, 1\le  k\le t\}\]
\[\cup \{(\omega^{i,j},\eta^{k,l}\mid 1\le i\le j\le 2n, 1\le  k< l\le t\}
  \cup \{\eta^{i,i}\mid 1\le i\le t-1\}\]
\end{lemma}

For a proposition $P$, we put $\delta_{P}=1$ when $P$ is true,
 and $\delta_{P}=0$, otherwise. Lemmas \ref{SE} and \ref{se} give the following Theorem.

\begin{theorem}\label{H2}
Suppose that $\h^{\lambda,\kappa,\mu}_{m,n,t}$ is the twisted Heisenberg Lie superalgebra spanned by $\{e_1,\ldots,e_{2m+2}\mid \omega_1,\ldots,\omega_{2n},\eta_1,\ldots,\eta_t\}$. Then

(1) the space $H^1(\h^{\lambda,\kappa,\mu}_{m,n,t})$ is spanned by  the classes of the cocycles
\[\{ e^{2m+2}\mid \eta^1,\ldots,\eta^t\},\]

(2)the space $H^2(\h^{\lambda,\kappa,\mu}_{m,n,t})$ is spanned by  the classes of the cocycles
$\bigcup^{5}_{i=1} A_i,$
where
\begin{eqnarray*}
&&A_1=\{\delta_{\lambda_i=\pm\lambda_j} (e^{i,j} -\lambda_i\lambda_j^{-1}e^{m+i,m+j}),\delta_{\lambda_i=\pm\lambda_j} (e^{i,m+j} +\lambda_i\lambda_j^{-1}e^{j,m+i})\mid 1\le i\le j\le m\},\\
&&A_2=\{\delta_{\lambda_i=\pm\kappa_j} (e^{m+i}\wedge\omega^j-\lambda_i\kappa_j^{-1}e^{i}\wedge \omega^{n+j}),\delta_{\lambda_i=\pm\kappa_j} (e^{i}\wedge\omega^j-\lambda_i\kappa_j^{-1}e^{m+i}\wedge \omega^{n+j})\mid \\
&&\qquad \ 1\le i\le m, 1\le j\le n\},\\
&&A_3=\{\delta_{\kappa_i=\pm\kappa_j} (\omega^{i,j}-\kappa_i\kappa_j^{-1} \omega^{n+i,n+j}),\delta_{\kappa_i=\pm\kappa_j} (\omega^{i,n+j}-\kappa_i\kappa_j^{-1} \omega^{j,n+i})\mid 1\le i\le j\le n\},\\
&&A_4=\{\eta^{i,i},\eta^{k,l}\mid 1\le i\le t-1, 1\le k<l\le t\},\\
&&A_5=\{e^{2m+2}\wedge\eta^{k}\mid 1\le k\le t\}.
\end{eqnarray*}
In particular, 
$\mathrm{sdim}\ H^2(\h^{\lambda,\kappa,\mu}_{m,n,t})=
(2\mathrm{Card}\{(i,j)\mid \lambda_i=\pm \lambda_j,1\le i<j\le m\}+2\mathrm{Card}\{(i,j)\mid \kappa_i=\pm \kappa_j,1\le i\le j\le n\}+\frac{t(t+1)}{2}+m-1,
2\mathrm{Card}\{(i,j)\mid \lambda_i=\pm \kappa_j,1\le i\le m, 1\le j\le n\}+t).$

\end{theorem}
\begin{proof}
A direct computation shows that
\[e_{2m+2}\cdot e^i=-\lambda_ie^{m+i},\ 
e_{2m+2}\cdot e^{m+i}=-\lambda_ie^i,\
e_{2m+2}\cdot \omega^j=-\kappa_j\omega^{n+j},
\]
\[
 e_{2m+2}\cdot \omega^{n+j}=-\kappa_j\omega^{j},\ \
e_{2m+2}\cdot \eta^k=0,
\]
where $1\le i\le m$, $1\le j\le n$ and $1\le k\le t$.
 By Lemmas \ref{SE} and \ref{se}, we have
\begin{equation*}
H^1(\h^{\lambda,\kappa,\mu}_{m,n,t}) =H^1(\h_{m,2n+t})^{e_{2m+2}}\oplus \F e^{2m+2}=\mathrm{span}\{ e^{2m+2}\mid \eta^1,\ldots,\eta^t\}.
\end{equation*}
By Lemma \ref{SE}, we have
\begin{equation}\label{1}
H^2(\h^{\lambda,\kappa,\mu}_{m,n,t}) =H^2(\h_{m,2n+t})^{e_{2m+2}}\oplus \left(\F e^{2m+2}\wedge \frac{H^{1}(\h_{m,2n+t})}{e_{2m+2}\cdot H^{1}(\h_{m,2n+t})}\right).
\end{equation}
Note that
$e_{2m+2}\cdot H^1(\h_{m,2n+t})=\mathrm{span}\{e^1,\ldots,e^{2m}\mid\omega^1,\ldots,\omega^{2n}\}$. We have
\begin{equation}\label{2}
\frac{H^{1}(\h_{m,2n+t})}{e_{2m+2}\cdot H^{1}(\h_{m,2n+t})}=\mathrm{span}\{\eta^1,\ldots,\eta^t\}.
\end{equation}
To compute
$H^2(\h_{m,2n+t})^{e_{2m+2}}$,
we denote by the following subspaces of $H^2(\h_{m,2n+t})$:
\begin{eqnarray*}
&&V_1=\mathrm{span}\{e^{i,j}, e^{m+i,m+j}\mid 1\le i<j\le m\}, \\
&&V_2=\mathrm{span}\{e^{i,m+j}, e^{j,m+i}\mid 1\le i\le j\le m\},\\
&&V_3=\mathrm{span}\{e^i\wedge \omega^{n+j}, e^{m+i}\wedge \omega^j\mid 1\le i\le m, 1\le j\le n\},\\
&&V_4=\mathrm{span}\{e^i\wedge \omega^j, e^{m+i}\wedge \omega^{n+j}\mid 1\le i\le m, 1\le j\le n\},\\
&&V_5=\mathrm{span}\{\omega^{i,j}, \omega^{n+i,n+j}\mid 1\le i\le j\le n\},\\
&&V_6=\mathrm{span}\{\omega^{i,n+j}, \omega^{j,n+i}\mid 1\le i\le j\le n\},\\
&&V_7=\mathrm{span}\{\eta^{i,i},\eta^{k,l}\mid 1\le i\le t-1, 1\le k<l\le t\},\\
&&V_8=\mathrm{span}\{e^{i}\wedge \eta^k, \omega^j\wedge \eta^k \mid 1\le i\le 2m, 1\le  j\le 2n, 1\le  k\le t\}.
\end{eqnarray*}
Then it makes a decomposition of the space $H^2(\h_{m,2n+t})$, which induces a decomposition of $H^2(\h_{m,2n+t})^{e_{2m+2}}$, that is, 
\[H^2(\h_{m,2n+t})^{e_{2m+2}}=\left(\bigoplus_{i=1}^{8}V_i\right)^{e_{2m+2}}=\bigoplus_{i=1}^{8} V_i^{e_{2m+2}}.\]
Moreover, by a direct computation, we get
\begin{eqnarray*}
&&V_1^{e_{2m+2}}=\mathrm{span}\{\delta_{\lambda_i=\pm\lambda_j} (e^{i,j} -\lambda_i\lambda_j^{-1}e^{m+i,m+j})\mid 1\le i<j\le m\}, \\
&&V_2^{e_{2m+2}}=\mathrm{span}\{\delta_{\lambda_i=\pm\lambda_j} (e^{i,m+j} +\lambda_i\lambda_j^{-1}e^{j,m+i})\mid 1\le i\le j\le m\}, \\
&&V_3^{e_{2m+2}}=\mathrm{span}\{\delta_{\lambda_i=\pm\kappa_j} (e^{m+i}\wedge\omega^j-\lambda_i\kappa_j^{-1}e^{i}\wedge \omega^{n+j})\mid 1\le i\le m, 1\le j\le n\}, \\
&&V_4^{e_{2m+2}}=\mathrm{span}\{\delta_{\lambda_i=\pm\kappa_j} (e^{i}\wedge\omega^j-\lambda_i\kappa_j^{-1}e^{m+i}\wedge \omega^{n+j})\mid 1\le i\le m, 1\le j\le n\}, \\
&&V_5^{e_{2m+2}}=\mathrm{span}\{\delta_{\kappa_i=\pm\kappa_j} (\omega^{i,j}-\kappa_i\kappa_j^{-1} \omega^{n+i,n+j})\mid 1\le i\le j\le n\}, \\
&&V_6^{e_{2m+2}}=\mathrm{span}\{\delta_{\kappa_i=\pm\kappa_j} (\omega^{i,n+j}-\kappa_i\kappa_j^{-1} \omega^{j,n+i})\mid 1\le i\le j\le n\}, \\
&&V_7^{e_{2m+2}}=V_7,\\
&&V_8^{e_{2m+2}}=0.
\end{eqnarray*}
Then we get
$H^2(\h_{m,2n+t})^{e_{2m+2}}$ is spanned by the set $\bigcup_{i=1}^4 A_i$.
The proof follows from Eqs. (\ref{1}) and (\ref{2}).
\end{proof}

\section{The cohomology $H_*^1(\h^{\lambda,\kappa,\mu}_{m,n,t})$ and $H_*^2(\h^{\lambda,\kappa,\mu}_{m,n,t})$}

 The
$[p]$-operator formula (\ref{res})  implies
\[
[\h^{\lambda,\kappa,\mu}_{m,n,t},\h^{\lambda,\kappa,\mu}_{m,n,t}]
+
\langle (\h^{\lambda,\kappa,\mu}_{m,n,t})_{\bar{0}}^{[p]}\rangle=\mathrm{span}\{e_1,\ldots,e_{2m+2}\mid \omega_1,\ldots,\omega_{2n}\}.
\]
 Then Eq. (\ref{h1})
follows that 
$H_*^1(\h^{\lambda,\kappa,\mu}_{m,n,t})=\mathrm{span}\{\eta^1,\ldots,\eta^t\}$. To compute the cohomology $H_*^2(\h^{\lambda,\kappa,\mu}_{m,n,t})$, we recall the Hochschild's six-term exact sequence for a restricted Lie superalgebra
$\g$, see \cite{H,Viv} for more details, which
relates the ordinary and restricted 1- and 2-cohomology spaces:
\begin{equation}\label{sixterm}
\begin{split} 
0&\longrightarrow H^1_*(\g)\stackrel{\iota}{\longrightarrow}  H^1(\g)\stackrel{D}{\longrightarrow}
\Hom_{\rm Fr}(\g_{\bar{0}},\F) \longrightarrow
H^2_*(\g)\stackrel{\pi}{\longrightarrow }H^2(\g)\longrightarrow\\
&\stackrel{H}{\longrightarrow} \Hom_{\rm
    Fr}(\g_{\bar{0}},H^1(\g)).
\end{split}
\end{equation}
The maps
$D: H^1(\g)\to\Hom_{\rm Fr}(\g_{\bar{0}},\F) $
and
$H: H^2(\g)\to\Hom_{\rm Fr}(\g_{\bar{0}},H^1(\g))$
in
(\ref{sixterm}) are  given by

\begin{equation}\label{D}
D_\psi(g)= \psi(g^{[p]})
  \end{equation}
and
\begin{equation}\label{HH}
H_\phi(g)\cdot h =\varphi(g\wedge (\ad g)^{p-1}(h))-\phi(g^{[p]}\wedge h)
  \end{equation}
where $g\in \g_{\bar{0}}$ and $h\in\g$. For $\g=\h^{\lambda,\kappa,\mu}_{m,n,t}$ and $\varphi\in \mathrm{ker}(H)$, by Theorem \ref{H2} (2), we can let $\varphi=\varphi_1+\sum_{k=1}^t a_k e^{2m+2}\wedge \eta^k$, where $\varphi_1\in\langle \bigcup_{i=1}^4 A_i\rangle$ and $a_k\in\F$. Then, for $1\le k\le t$, we have 
\begin{eqnarray*}
H_{\varphi}(e_{2m+2})\cdot    \eta_k=-\varphi(e_{2m+2}^{[p]}\wedge \eta_k)&=&-\varphi((|\lambda|e_{2m+2}+\mu_{2m+2} e_{2m+1})\wedge \eta_k),\\
&=&-\sum_{k=1}^t a_k (e^{2m+2}\wedge \eta^k)(|\lambda|e_{2m+2}\wedge \eta_k),\\
&=&0.
\end{eqnarray*}
From this, we get all $a_k=0$. So $\varphi=\varphi_1\in\langle \bigcup_{i=1}^4 A_i\rangle$. In contrast, for
$\varphi_1\in\langle \bigcup_{i=1}^4 A_i\rangle$ and $g=\sum^{2m+2}_{i=1} a_i e_i\in (\h^{\lambda,\kappa,\mu}_{m,n,t})_{\bar{0}}$,  we have
\[
H_{\varphi_1}(g)\cdot e_{2m+2}=\varphi_1(g\wedge (\ad g)^{p-1}(e_{2m+2}))=-a_{2m+2}^{p-2}|\lambda|\varphi_1(g,g)=0.
\]
Therefore, by Theorem \ref{H2} (1), $\mathrm{ker}(H)=\langle \bigcup_{i=1}^4 A_i\rangle$.
Note that $D(\eta^1)=\cdots=D(\eta^t)=0$, $D(e^{2m+2})=|\lambda|\overline e^{2m+2}$.
Then, by Theorem \ref{H2} (1), $\mathrm{im}(D)=\F \overline e^{2m+2}$.
Moreover, the sequence (\ref{sixterm}) reduces to the splitting exact sequence
\begin{equation}\label{Ex}
\begin{split} 
0&\longrightarrow
\Hom_{\rm Fr}((\h^{\lambda,\kappa,\mu}_{m,n,t})_{\bar{0}},\F)/ \F \overline e^{2m+2}\longrightarrow
H^2_*(\h^{\lambda,\kappa,\mu}_{m,n,t})\longrightarrow \left\langle \bigcup_{i=1}^4 A_i\right\rangle \longrightarrow 0,
\end{split}
\end{equation}
where the map $\Hom_{\rm Fr}((\h^{\lambda,\kappa,\mu}_{m,n,t})_{\bar{0}},\F)/\F \overline e^{2m+2}
\rightarrow H^2_*(\h^{\lambda,\kappa,\mu}_{m,n,t})$
sends $\overline e^i$ to the class of
$(0,\overline e^i)$  for $1 \le i \le 2m+1$.

\begin{theorem}\label{H*2}
Suppose that $\h^{\lambda,\kappa,\mu}_{m,n,t}$ is the restricted twisted Heisenberg Lie superalgebra spanned by $\{e_1,\ldots,e_{2m+2}\mid \omega_1,\ldots,\omega_{2n},\eta_1,\ldots,\eta_t\}$. Then
 the space $H^2_*(\h^{\lambda,\kappa,\mu}_{m,n,t})$ is spanned by  the classes of the cocycles
\[
\left\{(\varphi,\widetilde{\varphi})
   \mid \varphi\in\bigcup_{i=1}^4 A_i\right\}\cup
\{(0,\overline e^i)\mid 1\le i\le 2m+1\},
\]
where $A_i$ is defined in Theorem \ref{H2} (2) for $1\le i \le 4$. In particular, 
$\mathrm{sdim}\ H^2(\h^{\lambda,\kappa,\mu}_{m,n,t})=
(2\mathrm{Card}\{(i,j)\mid \lambda_i=\pm \lambda_j,1\le i<j\le m\}+2\mathrm{Card}\{(i,j)\mid \kappa_i=\pm \kappa_j,1\le i\le j\le n\}+\frac{t(t+1)}{2}+3m,
2\mathrm{Card}\{(i,j)\mid \lambda_i=\pm \kappa_j,1\le i\le m, 1\le j\le n\}).$
\end{theorem}
\begin{proof}
The exact sequence (\ref{Ex}) follows that
 \[
 H^2_*(\h^{\lambda,\kappa,\mu}_{m,n,t})\cong H^2(\h^{\lambda,\kappa,\mu}_{m,n,t})\oplus\Hom_{\rm Fr}((\h^{\lambda,\kappa,\mu}_{m,n,t})_{\bar{0}},\F)/\F \overline e^{2m+2}.
 \]
The proof follows from Theorem \ref{H2} (2).
\end{proof}

\section{Restricted one-dimensional central extensions of
  $\h^{\lambda,\kappa,\mu}_{m,n,t}$}

The definition of restricted central
extensions was introduced in \cite{BE}.
Let $(\g,[p])$
be a restricted Lie superalgebra and $M$ be a strongly abelian restricted Lie superalgebra (i.e., $[M, M]=0$ and $M_{\bar{0}}^{[p]}=0$). 
A \emph{restricted extension} of $\g$ by $M$ is a short exact sequence of restricted Lie superalgebras
\begin{equation*}
  0 \longrightarrow   M \stackrel{\iota} {\longrightarrow}\mathfrak{G} \stackrel{\pi}{\longrightarrow}  \g  \longrightarrow  0.
\end{equation*}
If $\iota(M)$ is contained in the center of $\mathfrak{G}$, then the extension is called \emph{central}. Two restricted central extensions
of $\g$ by $M$ are called \emph{equivalent} if  there exists a restricted Lie superalgebra homomorphism $\sigma: \mathfrak{G}_1\to \mathfrak{G}_2$ such that $\pi_2 \sigma=\pi_1$. 
Consider the 
restricted central extensions
of $\g$ by a one-dimensional space $\F c$.
If $(\phi,\omega) \in C^2_*(\g)_{\bar{0}}$ is a
restricted even 2-cocycle, then the
corresponding restricted
one-dimensional central extension $\mathfrak{G}=\g\oplus\F c$ has  the bracket and
$[p]$-operator defined for all $g,h\in\g$ and $g_0\in\g_{\bar{0}}$ by
\begin{align}\label{genonedimext}
  \begin{split}
  [g,h]_{\mathfrak{G}}&=[g,h]+ \phi(g\wedge h)c\\  g_0^{[p]_{\mathfrak{G}}}&=g_0^{[p]} + \omega(g_{0})c
  \end{split}
\end{align}
where $[\cdot,\cdot]$ and $[p]$ denote the 
bracket of $\g$ and $[p]$-operator of $\g$, respectively. Moreover, we have the following theorem.

\begin{theorem}\cite[Theorem 3.5.3]{BE}\label{ce}
Let $(\g,[p])$ be a restricted Lie superalgebra. Then $H_*^{2}(\g)_{\bar{0}}$ is in one to one correspondence with the equivalence classes of restricted one-dimensional central extensions of $\g$.
\end{theorem}
Notice that different from Lie algebras [\cite{F}, Chapter 1, Section 4.6] and restricted Lie algebras [\cite{H}, Theorem 3.3], only even cocycles can be used to construct restricted central extensions of restricted Lie superalgebras.

With the equations (\ref{genonedimext}) together with
Theorem~\ref{H*2} we can explicitly describe the restricted
one-dimensional central extensions of $\h^{\lambda,\kappa,\mu}_{m,n,t}$. 
Let $g=\sum a_ie_i+\sum b_i\omega_i+\sum c_i\eta_i$, $h=\sum a_i' e_i+ \sum b'_i\omega_i+\sum c'_i\eta_i$ and $g_{0}=\sum d_ie_i$ denote  arbitrary elements of
$\h^{\lambda,\kappa,\mu}_{m,n,t}$ and $(\h^{\lambda,\kappa,\mu}_{m,n,t})_{\bar{0}}$, respectively. For $1\leq k<l\leq 2m$, Eq. (\ref{starprop}) gives
\begin{align}\label{p}
\begin{split}
  \widetilde{e^{k,l}}
  (g_0)&= \widetilde{e^{k,l}}  \left(\sum_{i=1}^{2 m+1}d_i e_i+d_{2 m+2}e_{2 m+2}\right)\\
  &=2^{-1}d_{2 m+2}^{p-2}\sum_{i,j=1}^{2 m}d_id_j e^{k,l}([e_i,\underbrace{e_{2 m+2},\ldots,e_{2 m+2}}_{p-2}],e_j)\\
        &= -2^{-1}d_{2m+2}^{p-2}    \sum_{i=1}^{m}\sum_{j=1}^{2 m}\lambda_{i}^{p-2}(d_i d_j
       e^{k,l}(e_{m+i},e_j) \\
       &+d_{m+i}d_j e^{k,l}(e_{i},e_j)).
 \end{split}
\end{align}
Since $g_0\in(\h^{\lambda,\kappa,\mu}_{m,n,t})_{\bar{0}}$, Eq. (\ref{starprop}) implies $\widetilde{\omega^{i,j}}(g_0)=\widetilde{\eta^{s,t}}(g_0)=0$ for $1\le i\le j\le 2 n$ and $1\le k\le l\le t$.

If $1\le i<j\le m$ such that $\lambda_i=\lambda_j$ or $\lambda_i=-\lambda_j$, and $\mathfrak{G}_{i,j}=\h^{\lambda,\kappa,\mu}_{m,n,t}\oplus \F c$ denotes
the one-dimensional restricted central extension of $\h^{\lambda,\kappa,\mu}_{m,n,t}$
determined by the cohomology class of the restricted cocycle
$$ \left(e^{i,j}-\lambda_i\lambda^{-1}_je^{m+i,m+j},
  \widetilde{e^{i,j}}-\lambda_i\lambda^{-1}_j\widetilde{e^{m+i,m+j}}\right),$$
   then (\ref{genonedimext}) and (\ref{p}) give the bracket and
$[p]$-operator in $\mathfrak{G}_{i,j}$:
\begin{align*}
  \begin{split}
    [g,h]_{\mathfrak{H}_{i,j}}& =[g,h]+(a_i a'_j-a_ja'_i-\lambda_i\lambda^{-1}_ja_{m+i}a'_{m+j}+\lambda_i\lambda^{-1}_ja_{m+j}a'_{m+i})c;\\
    g_0^{[p]_{\mathfrak{H}_{i,j}}} & = g_{0}^{[p]}-2^{-1}d_{2 m+2}^{p-2}(\lambda^{p-2}_{i} d_{m+i}d_{j}-2\lambda_{j}^{p-2}d_{m+j}d_{i}
    +\lambda_i\lambda_j^{p-3}d_jd_{m+i}) c.
  \end{split}
\end{align*}

If $1\le i\le j\le m$ such that $\lambda_i=\lambda_j$ or $\lambda_i=-\lambda_j$,  and $\mathfrak{G}_{i,m+j}=
\h^{\lambda,\kappa,\mu}_{m,n,t}
\oplus \F c$ denotes
the one-dimensional restricted central extension of
$\h^{\lambda,\kappa,\mu}_{m,n,t}$
determined by the cohomology class of the restricted cocycle
$$(e^{i,m+j}+\lambda_i \lambda^{-1}_je^{j,m+i}, \widetilde{ e^{i,m+j}}+\lambda_i\lambda^{-1}_j\widetilde{e^{j,m+i}}),$$
then Eqs. (\ref{genonedimext}) and (\ref{p}) give the bracket and
$[p]$-operator in $\mathfrak{G}_{i,m+j}$:
\begin{align*}
  \begin{split}
    [g,h]_{\mathfrak{H}_{i,m+j}} & =[g,h]+(a_i a'_{m+j}-a_{m+j}a'_i-\lambda_i\lambda^{-1}_ja_{m+i}a'_{j}+\lambda_i\lambda^{-1}_ja_{j}a'_{m+i})c;\\
      g_0^{[p]_{\mathfrak{H}_{i,m+j}}} & = g_0^{[p]}-2^{-1}d_{2m+2}^{p-2}(\lambda^{p-2}_{i} d_{m+i}d_{m+j}-\lambda_{j}^{p-2}d_{j}d_{i}-\lambda_j^{p-2}d_i d_{j}\\
    &+\lambda_i\lambda_j^{p-3}d_{m+j}d_{m+i}) c.
  \end{split}
\end{align*}

If $1\le i\le j\le n$ such that $\kappa_i=\kappa_j$ or $\kappa_i=-\kappa_j$,  and $\mathfrak{H}_{i,j}=\h^{\lambda,\kappa,\mu}_{m,n,t}\oplus \F c$ denotes
the one-dimensional restricted central extension of $\h^{\lambda,\kappa,\mu}_{m,n,t}$
determined by the cohomology class of the restricted cocycle
$$(\omega^{i,j}-\kappa_i\kappa_j^{-1} \omega^{n+i,n+j}, \widetilde{\omega^{i,j}}-\kappa_i\kappa_j^{-1} \widetilde{\omega^{n+i,n+j}}),$$
then Eq. (\ref{genonedimext})  gives the bracket and
$[p]$-operator in $\mathfrak{H}_{i,m+j}$:
\begin{align*}
  \begin{split}
    [g,h]_{\mathfrak{H}_{i,n+j}} & =[g,h]-(b_ib'_j+b_j b'_i-\kappa_i\kappa^{-1}_j b_{n+i}b'_{n+j}-\kappa_i\kappa^{-1}_j b_{n+j}b'_{n+i})c;\\      g_0^{[p]_{\mathfrak{H}_{i,n+j}}} & = g_0^{[p]}.
  \end{split}
\end{align*}

If $1\le i\le j\le n$ such that $\kappa_i=\kappa_j$ or $\kappa_i=-\kappa_j$,  and $\mathfrak{H}_{i,n+j}=\h_m^{\lambda.\mu}\oplus \F c$ denotes
the one-dimensional restricted central extension of $\h^{\lambda,\kappa,\mu}_{m,n,t}$
determined by the cohomology class of the restricted cocycle
$$(\omega^{i,n+j}-\kappa_i\kappa_j^{-1} \omega^{j,n+i}, \widetilde{\omega^{i,n+j}}-\kappa_i\kappa_j^{-1} \widetilde{\omega^{j,n+i}}),$$
then Eq. (\ref{genonedimext}) gives the bracket and
$[p]$-operator in $\mathfrak{H}_{i,m+j}$:
\begin{align*}
  \begin{split}
    [g,h]_{\mathfrak{H}_{i,n+j}} & =[g,h]-(b'_ib_{n+j}+b'_{n+j}b_i-\kappa_i\kappa^{-1}_j b'_{n+i}b_{j}-\kappa_i\kappa^{-1}_j b'_{j}b_{n+i})c;\\
      g_0^{[p]_{\mathfrak{H}_{i,m+j}}} & = g_0^{[p]}.
  \end{split}
\end{align*}

If $1\le k\le l\le t$ and $\mathfrak{J}_{s,t}=\h^{\lambda,\kappa,\mu}_{m,n,t}\oplus \F c$ denotes
the one-dimensional restricted central extension of $\h^{\lambda,\kappa,\mu}_{m,n,t}$
determined by the cohomology class of the restricted cocycle
$(\eta^{k,l},\widetilde{\eta^{k,l}})$, then Eq. (\ref{genonedimext}) gives the bracket  and
$[p]$-operator 
\begin{align*}
  \begin{split}
    [g,h]_{\mathfrak{J}_{s,t}} & =[g,h]-(c'_kc_l+c'_l c_k)c;\\
    g_0^{[p]_{\mathfrak{J}_{s,t}}} & = g^{[p]}_0.
  \end{split}
\end{align*}

If $1\le i\le 2m+1$ and $\mathfrak{G}_i=\h^{\lambda,\kappa,\mu}_{m,n,t}\oplus \F c$ denotes
the one-dimensional restricted central extension of $\h^{\lambda,\kappa,\mu}_{m,n,t}$
determined by the cohomology class of the restricted cocycle
$(0,\overline e^i)$, then Eq.  (\ref{genonedimext}) gives the bracket  and
$[p]$-operator 
\begin{align*}
  \begin{split}
    [g,h]_{\mathfrak{H}_i} & =[g,h];\\
    g_0^{[p]_{\mathfrak{H}_i}} & = g^{[p]}_{0}+ d_i^p c.
  \end{split}
\end{align*}
The central extensions $\mathfrak{G}_i$ form a basis for the $(2m+1)$-dimensional
space of restricted one-dimensional central extensions that split as
ordinary Lie superalgebra extensions (c.f. \cite{EFY,Y2}).
\\

\small\noindent \textbf{Acknowledgment}\\
The work is supported by the Talent Project of the Tianchi Doctoral Program in Xinjiang Uygur Autonomous Region (Grant No. 5105250184d) 
and Natural Science Foundation of Xinjiang Uygur Autonomous Region (Grant No. 2025D01C279).

\end{document}